\documentclass[11pt, a4paper]{article}
\usepackage{amsmath, amssymb, amsthm, verbatim,enumerate,bbm,color, hyperref}
\usepackage{mathrsfs}
\usepackage{geometry}

\newtheorem{theorem}{Theorem}
\newtheorem{lemma}[theorem]{Lemma}
\newtheorem{prop}[theorem]{Proposition}

\newtheorem*{definition}{Definition}
\newtheorem{example}{Example}
\newtheorem*{remark}{Remark}

\newcommand{\C}{\mathcal{C}}

\newcommand{\F}{\mathcal{F}}
\newcommand{\K}{\mathcal{K}}

\begin{document}

\title{Quasirandom Graphs and the Pantograph Equation}
\author{Asaf Shapira \thanks{
		School of Mathematics, Tel Aviv University, Tel Aviv 69978, Israel.
		Email: asafico$@$tau.ac.il 
	} \and
	Mykhaylo Tyomkyn
	\thanks{Department of Applied Mathematics, Charles University, Czech Republic. Email: tyomkyn$@$kam.mff.cuni.cz
	}}

\date{}
\renewcommand{\baselinestretch}{1.05}

\maketitle

\begin{abstract}
The pantograph differential equation and its solution, the deformed exponential function, are remarkable objects that appear in areas as diverse as combinatorics, number theory, statistical mechanics, and electrical engineering. In this article we describe a new surprising application of these objects in graph theory, by showing that the set of all cliques is not forcing for quasirandomness. This provides a natural example of an infinite family of graphs, which is not forcing, and answers a natural question posed by P.~Horn.
%
\end{abstract}

\section{The pantograph equation and its solution.}

The best known differential equation 
$$y'(x)=py(x), \ y(0)=1
$$
is solved by the exponential function $y=e^{px}$. Consider a visually similar but very different equation 
$$y'(x)=y(px), \ y(0)=1,
$$
where $0<p\leq 1$.
This is a special case of the \emph{pantograph} equation, one of the most studied examples of \emph{delay differential equations}, where the value of the derivative of $y$ at time $x$ is a function of the value of $y$ at an earlier time $px$. The pantograph equation owes its name to a component of the electric locomotive, connecting it to the overhead wire, and was first studied in the late 1960's by British Railways. The equation results from the study of the pantograph's movement, where it is crucial that the pantograph stays in constant contact with the wire, in order to collect current for the locomotive without interruptions; see~\cite{Gr} for more historical and physics background.

The pantograph equation and its generalizations have found many applications in physics~(see \cite{Liu}) and mathematics. It can be verified that its unique solution is the so-called \emph{deformed exponential function} 
\begin{equation}\label{eq:defexp}
f_p(x)=\sum_{j=0}^\infty \frac{x^j}{j!}p^{\binom{j}{2}}=1+x+\frac{x^2}{2}p+\frac{x^3}{6}p^3+\cdots;
\end{equation}
note that when $p=1$ we have $f_1(x)=e^x$, hence the name. In fact, the deformed exponential was first studied by Mahler in 1940~\cite{Ma}, nearly 30 years before the pantograph equation made an appearance. Mahler's motivation was in number theory: he used the function $f_p$ to derive an asymptotic formula for the number of partitions of a large integer $n$ into powers of a fixed integer $r$. Later the deformed exponential was found to appear naturally in many other contexts. In combinatorics it is connected to the Tutte polynomial of complete graphs~\cite{Tutte}, the enumeration of acyclic digraphs~\cite{Ro}, and inversions of trees~\cite{MaRi}. In statistical mechanics the function $f_p$ appears as the partition function of one-site lattice gas~\cite{ScSo}, and in complex analysis it is related to the Whittaker and Goncharov constants~\cite{Bo}. 
The function $f_p$ continues to be the focus of current research; see for example the recent paper~\cite{WZ} which studies the asymptotics of its roots.

In this article, we present an unexpected application of certain properties of the deformed exponential function to the theory of quasirandom graphs.

\section{Quasirandom graphs and forcing families.}

\emph{Quasirandomness}, or \emph{pseudorandomness}, is a phenomenon occurring in several areas of discrete mathematics: number theory, group theory, combinatorics, and graph theory. It can be loosely described as the study of properties of truly random objects
in deterministic ones; see~\cite{Tao} for a general survey, and \cite{KrSu} for a survey on pseudorandom graphs.

Let us first focus on counting copies of a fixed small graph $H$ inside a large graph $G$. It will be more convenient to count \emph{labeled} copies, that is, injective mappings from the vertex set of $H$ to that of $G$ that map edges to edges. Let us illustrate this with some examples.

\begin{example}\label{ex:stars}
If $G$ is a complete bipartite graph\footnote{A complete bipartite graph is a graph on the vertex set $A \cup B$, where $A$ and $B$ are disjoint, and whose edge set is $A \times B$. More generally, a complete $k$-partite graph is a graph whose vertex set is composed of $k$ disjoint sets $A_1,\ldots,A_k$ and whose edge set is $\bigcup_{i < j}A_i \times A_j$. In other words, every two vertices in distinct $A_i, A_j$ are connected by an edge.
}  with both parts of size $n/2$, where $n$ is a large even number, and $H$ is a star with $k$ edges, i.e., a complete bipartite 
graph with part sizes $1$ and $k$, then $G$ contains $$n\cdot \frac{n}{2}\left(\frac{n}{2}-1\right)\cdots\left(\frac{n}{2}-k+1\right)=(1+o(1))2^{-k}n^{k+1}$$ labeled copies of $H$.	
\end{example}	

In the above example, as well as throughout the rest of the article, any $o(1)$ expression should be read as a function tending to $0$ as $n$ goes to
$\infty$. The $o()$-notation is an equivalent, yet more convenient-to-use form of the usual $\epsilon$-$\delta$-$n_0$ formalism. In particular, we can use multiple $o(1)$-expressions in the same formula, saving us the need to introduce multiple $\epsilon$'s (two $o(1)$-expressions are not assumed to be identical).
Note that here and later the $o()$-notation assumes that $n$ tends to infinity and treats all other variables (e.g., $k$ in Example~\ref{ex:stars}) as constant parameters.


\begin{example}\label{ex:cycles}
If $G$ is a clique\footnote{A clique or complete graph is a graph on a vertex set $V$, whose edge set consists of all pairs $\{u,v\}\subseteq V$.} on $pn$ vertices, where $n$ is large and $0<p\leq 1$, and $H$ is a cycle of length $k$, then $G$ has 
$$pn(pn-1)\cdots (pn-k+1)=(1+o(1))p^kn^k
$$  
labeled copies of $H$.
\end{example}

For $0<p<1$ and a large integer $n$, the \emph{$p$-random graph} on $n$ vertices~\cite{FK}, also known as the binomial random graph and denoted by $G(n,p)$, is obtained by taking $n$ labeled vertices and including every edge between them randomly and independently with probability $p$. Then, a standard probabilistic argument using Chebyshev's inequality (also known as the second moment method~\cite{FK}) implies the following.

\begin{example}\label{ex:random}
For any fixed graph $H=(V,E)$, the random graph $G(n,p)$ contains with high probability (i.e., with probability tending to $1$ as $n$ grows) $$(1+o(1))p^{|E|}n^{|V|}$$ labeled copies of $H$. In particular, $G(n,p)$ with high probability contains
$(1+o(1))pn^2$ labeled edges and $(1+o(1))p^4n^4$ labeled copies of the $4$-cycle $C_4$. 
\end{example}

The reason we singled out the edges and $C_4$’s in the last example is the following seminal
result of Chung, Graham, and Wilson~\cite[Theorem 1]{CGW}. In what follows, when speaking about ``large graphs," we mean,
formally, sequences of graphs with the number of vertices $n$ tending to $\infty$.

\begin{theorem}[\cite{CGW}]\label{thm:CGW}
	Suppose that $0<p<1$. The following properties of (large) $n$-vertex graphs $G$ are equivalent. 
	\begin{enumerate}
		\item[(P1)] $G$ has $(1+o(1))pn^2$ labeled edges and 
		$$(1+o(1))p^4n^4$$
		labeled copies of $C_4$.
		\item[(P2)] For every fixed graph $H=(V,E)$, the number of labeled copies of $H$ in $G$ is  
		$$(1+o(1))p^{|E|}n^{|V|}.$$
		\item[(P3)] For every $c>0$ and vertex set $S\subseteq V(G)$ of size $|S|\geq cn$, the number of labeled edges between vertices in $S$ is $$(1+o(1))p|S|^2.$$
	\end{enumerate}
\end{theorem}
Observe that, by Example~\ref{ex:random}, a random graph $G(n,p)$ satisfies with high probability property (P2) and thus, a fortiori, also (P1).
Similarly, a standard application of the Chernoff  bound~\cite{FK} shows that $G(n,p)$ satisfies with high probability property (P3).
The remarkable aspect of Theorem~\ref{thm:CGW} is that every (deterministic) graph that satisfies the seemingly very weak property (P1) must
also satisfy the much stronger properties (P2) and (P3). In fact, the main result of \cite{CGW} exhibited a number of further equivalent 
conditions (which we do not state here formally, for brevity).

A large graph satisfying either (P1), (P2), or (P3) (and therefore all three of them), is called \emph{$p$-quasirandom}. A graph is \emph{quasirandom} if it is $p$-quasirandom for some $0<p<1$. The notion of quasirandomness is central to extremal combinatorics: for instance, Szemer\'edi's famous regularity lemma~\cite{Sz}  states (vaguely speaking) that the vertices of a large graph can be partitioned into a bounded number of parts, so that ``almost all" bipartite graphs between those parts are quasirandom, i.e., resemble a truly random subgraph of a complete bipartite graph.

A very sensible question to ask at this point is, whether any deterministically constructed quasirandom graphs are known to exist. The answer is yes, and one prominent class of examples are the \emph{Paley graphs} (arising from a similar construction for matrices in~\cite{Pa}). Their quasirandomness can be deduced from properties of quadratic residues.
\begin{example}\label{ex:paley}
Let $n=4k+1$ be a prime, so that $x$ is a quadratic residue modulo $n$ if and only if $-x$ is one. Let $G$ be a graph on the vertex set $\{0,\dots,n-1\}$, where $xy$ is an edge whenever $x-y$ is a quadratic residue modulo $n$. Then $G$ is $1/2$-quasirandom.
\end{example}
On the cautious side we would like to add that some properties of the truly random graph $G(n,p)$ are not captured by quasirandomness. For example, the largest clique in $G(n,p)$ is with high probability of order $\log n$, whereas in quasirandom graphs it can be of ``almost linear" size.  
 
The discussion above led the authors of~\cite{CGW} to define the notion of a \emph{forcing} graph family.

\begin{definition}
	A family of graphs $\F$ is \emph{forcing} if the following holds for every $0<p<1$. Suppose $G$ is an $n$-vertex graph so that for every $F\in \F$ the graph $G$ contains $(1+o(1))n^{|V(F)|}p^{|E(F)|}$ labeled copies of $F$. Then $G$ is $p$-quasirandom. 
\end{definition} 

By this definition, Theorem~\ref{thm:CGW} states that the pair $\{K_2,C_4\}$ is forcing. To obtain a better understanding of the quasirandomness phenomenon, one should look for a classification of forcing families. This natural inverse question to Theorem~\ref{thm:CGW} was raised by Chung, Graham, and Wilson in the same paper~\cite{CGW}.

In the subsequent years, this topic has seen a significant amount of research (see~\cite{HPS} and the references therein), resulting in discoveries of a number of further forcing families. It is well known that for any nonbipartite graph $H$ the pair $\{K_2,H\}$ is not forcing, and the main open question in this area is the \emph{forcing conjecture} by Skokan and Thoma~\cite{SkTh}, saying that for every connected bipartite graph $H$ that is not a tree the pair $\{K_2,H\}$ is forcing. In~\cite{SkTh} this was proved for every complete bipartite graph $H$ (more generally, it was shown in~\cite{SkTh} that $\{H_1,H_2\}$ is forcing for any pair of distinct complete bipartite graphs), but the full conjecture is still wide open.

In the light of the above, it is even more challenging to decide whether an \emph{infinite} family of graphs is forcing. Consider, for instance, four of the most natural such families, namely the sets of all cycles, stars, trees, and cliques. It is easy to see that the family of all cycles is not forcing for all $0<p<1$, since, by Example~\ref{ex:cycles}, the graph comprising a clique on $pn$ vertices and $(1-p)n$ isolated vertices has the ``correct" number $(1+o(1))p^\ell n^\ell$ of labeled cycles of length $\ell$, but fails to satisfy (P3) of Theorem~\ref{thm:CGW}. Similarly, the family of all stars is not forcing by Example~\ref{ex:stars} --- in fact, this example shows that the family of all trees is not forcing.

The situation is less clear for the set of all finite cliques $\K=\{K_2,K_3,\dots\}$, where $K_j$ denotes the complete graph on $j$ vertices. Horn~\cite{West} asked
whether $\K$, or perhaps some finite subset of $\K$, is forcing (the latter would clearly imply the former) --- this would mean that any large graph having $(1+o(1))p^{\binom{j}{2}}n^j$ labeled copies of $K_j$ for each $j\geq 2$ satisfies property (P3) of Theorem~\ref{thm:CGW}. Both questions have been open until now, and our aim in this article is to answer them in the negative. 

First we give an elementary proof of the fact that any set of finitely many cliques is not forcing.  

\begin{theorem}\label{thm:main}
For any $k\geq 2$ and $0<p\leq 1/4$ there exist arbitrarily large $n$-vertex graphs $G_{k,p}(n)$ with $(1+o(1))p^{\binom{j}{2}}n^j$ labeled copies of $K_j$ for all $j=2,\dots,k$, and an independent set of size at least $n/2$. Therefore, the family $\K_k=\{K_2,\dots, K_k\}$ is not forcing.

\end{theorem}
The theorem above deals only with finitely many cliques and only with $p\leq 1/4$.\footnote{Strictly speaking, in order to prove that a family $\F$ is not forcing it is enough to exhibit a counterexample for just one $0<p<1$. Still, it is more desirable to give examples for all $0<p<1$.} By applying some properties of the deformed exponential function defined in~\eqref{eq:defexp}, we prove our main result, extending Theorem~\ref{thm:main} to all cliques and to all values of $p$.
\begin{theorem}\label{thm:allp}
The infinite family $\K=\{K_2,K_3,\dots\}$ is not forcing for any $0<p<1$.
\end{theorem}
\noindent

\section{Proofs of the main results.}

\subsection{The finite case.}
First we deal with the case of finitely many cliques.
\begin{proof}[Proof of Theorem~\ref{thm:main}.]
We claim that for fixed $k\geq 2$ and $0<p\leq 1/4$ there exist $k$ real nonnegative numbers $(c_1,\dots,c_k)=(c_1(k,p),\dots,c_k(k,p))$ with the following properties.
\begin{enumerate}
\item[(i)] For all $1\leq j\leq k$, $$\sum_{A\in\binom{[k]}{j}}\prod_{i\in A}c_i=\frac{p^{\binom{j}{2}}}{j!},$$
where $\binom{[k]}{j}$ stands for the set of all $j$-element subsets of $\{1,\dots,k\}$.
\item[(ii)] $\max c_i\geq 3/4$.
\end{enumerate}	
Given these numbers, define $G=G_{k,p}(n)$ to be a graph on $n$ vertices as follows. Partition $V(G)$ into sets $V_1,\dots, V_k$, such that for all $1\leq i\leq k$ we have $||V_i|-c_in|< 1$; this is possible since, by (i), $\sum_{i=1}^k c_i=1$. Let $E(G)$ be the set of all edges $uv$ where $u\in V_i, v\in V_j$ for $i\neq j$. In other words, $G$ is the complete $k$-partite graph on $(V_1,\dots,V_k)$.

By this construction, for every $j=2,\dots,k$ the graph $G_{k,p}(n)$ has
\begin{align*}
j!\sum_{A\in \binom{[k]}{j}}\prod_{i\in A}|V_i| &=j!\sum_{A\in \binom{[k]}{j}}\prod_{i\in A}(c_i n\pm 1)=(1+o(1))j!\sum_{A\in \binom{[k]}{j}}\prod_{i\in A}c_i n^j\\
&=(1+o(1))p^{\binom{j}{2}}n^j
\end{align*}
labeled copies of $K_j$. Here $c_in\pm 1$ stands for an integer within $1$ of $c_i n$. On the other hand, by (ii), it has an independent set of size at least $n \cdot \max c_i -1\geq n/2$. Hence, this graph satisfies the assertions of Theorem~\ref{thm:main}. Note that the fact that $G_{k,p}(n)$ has an independent set of size at least $n/2$ implies it fails property (P3) of Theorem~\ref{thm:CGW} for $c=1/2$. Therefore, $G_{k,p}$ is not $p$-quasirandom for any $0<p<1$, and we conclude that $\K_k$ is not forcing.

It remains to construct the sequence $(c_1,\dots,c_k)$ with the properties above. To this end, consider the real polynomial function
\begin{equation}\label{eq:trunc} 
f_{p,k}(x)=\sum_{j=0}^k \frac{{p^{\binom{j}{2}}}}{j!}x^j,
\end{equation}
which is a truncated version of the deformed exponential function~\eqref{eq:defexp}.
 
Kurtz~\cite{K} established the following useful criterion for a polynomial to have only real roots, which can be viewed as a converse to Newton's inequalities.
	\begin{prop}[\cite{K}, Theorem 2]\label{prop:Kurtz}
		If the coefficients of a real polynomial $P(x)=\sum_{i=0}^n b_ix^i$ satisfy $b_i>0$ for all $i$ and $b_i^2>4b_{i-1}b_{i+1}$ for all $1\leq i\leq n-1$, then all the roots of $P$ are real and distinct.\footnote{Note that the constant $4$ is best possible for $n=2$}
	\end{prop}
	\noindent
	In order to apply Proposition~\ref{prop:Kurtz} to the polynomial $f_{p,k}(x)$ all we need to check is that 
	\begin{equation*}
		\frac{p^{2\binom{j}{2}}}{j!j!}>\frac{4p^{\binom{j+1}{2}+\binom{j-1}{2}}}{(j+1)!(j-1)!}
	\end{equation*}
	holds for every $j=1,\dots,k-1$. This simplifies to 
	$$\frac{1}{j}>\frac{4p}{j+1},
	$$
	which is evidently true for all $p\leq 1/4$.
	Thus, $f_{p,k}$ has $k$ distinct real roots $a_1,\dots,a_k$, and can be written as 
	\begin{equation}\label{eq:factorize}
	f_{p,k}(x)=\frac{p^{\binom{k}{2}}}{k!}\prod_{i=1}^k(x-a_i).
	\end{equation}
	Since the coefficients of $f_{p,k}$ are positive, $f_{p,k}(x)>0$ for $x\geq 0$, and thus all of the roots are negative. Writing $c_i:=-1/a_i$ (so that $a_i=-1/c_i$), we obtain
	\begin{align}\label{eq:long}
	f_{p,k}(x)&=\frac{p^{\binom{k}{2}}}{k!}\prod_{i=1}^k\left(x+\frac{1}{c_i}\right)=\frac{p^{\binom{k}{2}}}{k!}\left(\prod _{i=1}^k c_i\right)^{-1}\prod_{i=1}^k(1+c_ix)\nonumber \\
	&=\frac{p^{\binom{k}{2}}}{k!}(-1)^k\prod _{i=1}^k a_i\prod_{i=1}^k(1+c_ix).
	\end{align}
	We will now show that the above-defined $c_1,\dots,c_k$ satisfy the properties (i) and (ii) stated at the beginning of the proof.
	Evaluating the constant term in~\eqref{eq:factorize} gives 
	$$\frac{p^{\binom{k}{2}}}{k!}(-1)^k\prod _{i=1}^k a_i=1,
	$$
	which, combined with~\eqref{eq:long}, implies
	\begin{equation}\label{eq:product}
f_{p,k}(x)=\prod_{i=1}^k(1+c_ix).
	\end{equation}
Next, evaluating in~\eqref{eq:product} the coefficient of $x^j$ for all $1\leq j\leq k$ gives
\begin{equation}\label{eq:jcliques}
\sum_{A\in \binom{[k]}{j}}\prod_{i\in A}c_i =\frac{p^{\binom{j}{2}}}{j!},
\end{equation}
establishing property (i). In particular,~\eqref{eq:jcliques} implies $\sum_{i=1}^kc_i=1$ and
$\sum_{1\leq i< j\leq k} c_ic_j=p/2$. Therefore,
	\begin{align*}
		\max c_i&=\max c_i\cdot \sum_{i=1}^k c_i \geq \sum_{i=1}^k c_i^2 
		=\left(\sum_{i=1}^k c_i\right)^2-2\sum_{1\leq i< j\leq k} c_ic_j=1-p
		\geq \frac{3}{4},
	\end{align*}
establishing property (ii).	
\end{proof}
It turns out that the polynomial $f_{p,k}$ has imaginary roots when $p > 1/2$; hence the approach we used in the proof above cannot cover all $0<p<1$. Therefore, to extend this to all $0<p<1$ and to handle the set of all cliques, we need a slightly different approach. 
\subsection{The general case.} We now prove that $\K$, the family of all finite cliques, is not forcing. As a first step towards the proof of Theorem~\ref{thm:allp}, we will construct an ``infinite graph" satisfying its assertion. We briefly mention that in the theory of graph limits (see \cite{Lo}) such an object is called a {\em graphon}. 

\begin{lemma}\label{lem:graphons}
For every $0<p<1$ there is a graph $W_p$ whose vertex set is the interval $[0,1]$ and which satisfies:
\begin{enumerate}
	\item[(1)] For every $k\geq 2$, if we randomly and (Lebesgue-)uniformly select $k$ vertices, $v_1,\dots,v_{k}$ from $[0,1]$, 
	then the probability that for every $i<j$ the vertices $v_i,v_j$ are connected by an edge in $W_p$ is $p^{\binom{k}{2}}$.
	\item[(2)] There is an interval $I \subseteq [0,1]$ of length $1-p$ so that $\{x,y\}$ is not an edge of $W_p \ $ for every $x,y \in I$. 
\end{enumerate}
\end{lemma}
Observe that assertion (1) is a continuous counterpart to the property of a finite $n$-vertex graph containing the ``correct" number of cliques. Similarly, assertion (2) states that $W_p$ has an independent set on a $(1-p)$-fraction of its vertices --- a finite graph with this property would fail to satisfy (P3) of Theorem~\ref{thm:CGW}.
\begin{proof}[Proof of Lemma~\ref{lem:graphons}]
For a sequence of positive real numbers $\C=(c_1,c_2,\dots)$ and an integer $k\geq 1$, we use  $\sigma_k(\C)$ to denote the formal expression $\sum_{A\in \binom{\mathbb{N}}{k}}\prod_{j\in A}c_j$. Similarly to the proof of Theorem~\ref{thm:main}, we claim that for every $0<p<1$ there exists a sequence $\C=(c_1,c_2,\dots)$ of positive reals $1> c_1\geq c_2\geq\dots >0$ with the following properties. 

\begin{enumerate}
\item[(i)] For each $k\geq 1$, $\sigma_k(\C)$ is convergent, with $\sigma_k(\C)=p^{\binom{k}{2}}/k!$. 
\item[(ii)] $\max c_i =c_1 \geq 1-p.$
\end{enumerate}

With such a sequence at hand, we partition the $[0,1]$-interval into infinitely many intervals $V_1,V_2,\dots$, such that $|V_i|=c_i$ for all $i$ (note that $\sum_{i=1}^{\infty}c_i=\sigma_1(\C)=1$), and we let $W_p$ be the graph on the vertex set $[0,1]$, where $x$ and $y$ are connected by an edge if and only if they belong to different intervals $V_i$ and $V_j$.

Then for every fixed $k\geq 2$, a random sample of $k$ vertices $v_1,\dots,v_k$ from $[0,1]$ forms a clique in $W_p$ with probability
$k!\sigma_k(\C)=p^{\binom{k}{2}}$. Moreover, for the interval $V_1$ we have $|V_1|=c_1\geq 1-p$, and no $x,y\in V_1$ are connected by an edge in $W_p$. Hence, $W_p$ satisfies both assertions of the lemma.

To construct the desired sequence $\C$, we consider the deformed exponential function $f_p(z)$ over a complex variable $z$, defined in~\eqref{eq:defexp}. It was shown in~\cite{Is} and~\cite{MFB} that for every $0<p<1$ the function 
 $f_p$ can be represented as a product: 
\begin{equation}\label{eq:hadamard}
f_p(z)=\prod_{i=1}^{\infty} \left(1-\frac{z}{a_i}\right),
\end{equation} 
where the $a_i$, the roots of $f_p$, are real (which means they must be negative, as the power series of $f_p$ has only positive signs).\footnote{To be more detailed, $f_p(z)$ is an \emph{entire function}, that is, $f_p$ is defined on all of $\mathbb{C}$, and is holomorphic everywhere (for more background information about entire functions see~\cite{Le, StS}). The \emph{order} of an entire function $g(z)=\sum_{j=0}^{\infty}a_jz^j$ is given by the formula $\rho =\limsup_{n\rightarrow \infty}\frac{n \log n}{\log (1/|a_n|)}$ (\cite[Section 1.3, Theorem 2]{Le}).
Thereby, the deformed exponential $f_p$ is of order $0$. By Hadamard's factorization theorem (\cite[Theorem 5.1]{StS}) an entire function of order $0$, taking value $1$ at $x=0$, can be represented as a product~\eqref{eq:hadamard}, where the $a_i$ are its complex roots. In~\cite{Is, MFB} it was shown that for any $0<p<1$ all roots of $f_p$ are real.}	
Thus, we can set $c_i=-1/a_i$ and $\C:=(c_1,c_2,\dots)$, where the elements are indexed in descending order, to obtain 
$$f_{p}(z)=\prod_{i=1}^{\infty}(1+c_iz).
$$ 
By comparing the coefficients of $z^k$ for every $k\geq 1$ we obtain that each $\sigma_k(\C)$ is convergent, with $\sigma_k(\C)=p^{\binom{k}{2}}/k!$.  
Furthermore, since all $c_i$ are positive, and $\sigma_1$ and $\sigma_2$ are (absolutely) convergent, for $c_1=\max \C$ we get
\begin{align*}
c_1&= c_1\cdot 1 = c_1 \cdot \sigma_1(\C) \geq \sum_{i=1}^\infty c_i^2
=(\sum_{i=1}^\infty c_i)^2-2\sum_{1\leq i<j}c_ic_j\\
&=\sigma_1(\C)^2-
2\sigma_2(\C)=1-p.
\end{align*}
\end{proof}
Lemma~\ref{lem:graphons} gives an example of a ``complete infinite-partite" graph that contains for every $j \geq 2$ the same fraction of
copies of $K_j$ as the random graph $G(n,p)$. Such a construction of course cannot be achieved for finite graphs. Instead, for every $k\geq 2$ we show how to turn $W_p$ into large graphs containing the correct number of labeled copies of $K_j$ for every $j \leq k$. 

\begin{proof}[Proof of Theorem~\ref{thm:allp}]
Take the infinite graph $W_p$ defined in Lemma~\ref{lem:graphons}, select $n$ vertices $v_1,\dots,v_n$ from it uniformly at random, and let $G_{p}(n)$ be the graph induced on these vertices. That is, the edges of $G$ are the edges of $W_p$ between $v_1,\dots, v_n$. 

Fix $k\geq 2$ and $\epsilon>0$. By assertion (1) of Lemma~\ref{lem:graphons} and the law of large numbers, for a sufficiently large $n= n(k,\epsilon)$ the graph $G_p(n)$ will, with probability greater than $1/2$, contain between $(1-\epsilon)n^jp^{\binom{j}{2}}$ and $(1+\epsilon)n^jp^{\binom{j}{2}}$ labeled copies of $K_j$ for all $j=2,\dots,k$. Additionally, assertion (2) of Lemma~\ref{lem:graphons} together with the law of large numbers imply that, for large $n$, with probability greater than $1/2$, $G_p(n)$ will have an independent set of size at least $(1-p)n/2$. 
Thus, with positive probability, there exists a graph $G_{k,p,\epsilon}(n)$ that has both properties. In other words, $G_{k,p,\epsilon}(n)$ is an 
$n$-vertex graph containing between $(1-\epsilon)n^jp^{\binom{j}{2}}$ and $(1+\epsilon)n^jp^{\binom{j}{2}}$ labeled $j$-cliques for $j=2,\dots,k$ and an independent set of size at least $(1-p)n/2$.  

Now, select a sequence $\epsilon_1>\epsilon_2>\dots>0$ such that $\lim_{k\rightarrow \infty} \epsilon_k=0$, for instance, $\epsilon_k=1/k$. For each $k$ take a graph $G_k=G_{k,p,\epsilon_k}(n)$ (for some $n$), and consider the sequence $G_2,G_3,\dots$; for the order $n=|V(G_k)|$ of the graphs we have $n\geq k$ (as $G_k$ contains a $k$-clique), so $n$ tends to $\infty$. Moreover, by construction, for every $j\geq 2$ the graphs in the sequence contain $(1+o(1))p^{\binom{j}{2}}n^j$ labeled copies of $K_j$, while also containing an independent set of size at least $(1-p)n/2$. In particular, $G=G_k$ fails property (P3) of Theorem~\ref{thm:CGW}, implying that $G$ is not $p$-quasirandom. Therefore, the set of all cliques $\K=\{K_2,K_3,\dots\}$ is not forcing for any $0<p<1$.
\end{proof}
\begin{remark}
The proof of Lemma~\ref{lem:graphons} relied on the fact that the roots of $f_p$ are all real. It is worth noting that a lot more is known about these numbers; for instance, the $k$th largest root (denoted $a_k$ above) is known to be of order $-kp^{1-k}$. Wang and Zhang~\cite{WZ} very recently established the asymptotics of the roots of $f_p$ up to arbitrary lower order terms.
\end{remark}

\section*{Acknowledgments.}
We thank two anonymous referees for their helpful remarks. The first author was supported in part by ISF Grant 1028/16, ERC Consolidator Grant 863438 and NSF-BSF Grant 20196.
The second author was supported in part by ERC Synergy grant DYNASNET 810115 and the H2020-MSCA-RISE project CoSP- GA No. 823748.

\end{document}